\documentclass[10pt,a4paper,reqno]{amsart}
\usepackage{amsfonts,amsthm,latexsym,amsmath,amssymb,amscd,amsmath}
\usepackage{graphicx}

\newtheorem{theorem}{Theorem}
\newtheorem{proposition}[theorem]{Proposition}
\newtheorem{lemma}[theorem]{Lemma}

\newtheorem{conjecture}{Conjecture}

\theoremstyle{definition}
\newtheorem{definition}{Definition}

\theoremstyle{remark}

\newtheorem{Example}{Example}

\newtheorem{Remark}{Remark}

\newcommand{\name}[1]{\operatorname{\mathrm #1}}
\newcommand \bydef {\mathrel{:=}}

\def \lmod#1\rmod {\vphantom{#1}\left|\smash{#1}\right|}

\newcommand{\cD}{\mathcal{D}}
\newcommand{\cP}{\mathcal{P}}
\newcommand{\cN}{\mathcal{N}}
\newcommand{\Complex}{\mathbb{C}}
\newcommand {\al}{\alpha}
\newcommand {\ga}{\gamma}
\newcommand{\DT}[1]{#1 \dots #1}

\newcommand{\ZZ}{\mathbb{Z}}
\newcommand{\C}{\mathcal{C}}
\newcommand{\CP}{\mathbb{CP}}
\newcommand{\OO}{\mathcal{O}}
\newcommand{\OOO}{\mathfrak{O}}
\newcommand{\Res}{\name{Res}}
\def \W {\mathcal W}
\def \h {\mathfrak H}
\def \tW {\widetilde {\mathcal W}}
\def \RemT {\mathcal{RT}}
\def \LocT {\mathcal{LT}}
\def \Node {\mathcal{ND}}
\def \Branch {\operatorname{\mathrm Br}}

\begin{document}
\numberwithin{equation}{section}

\title[On Hurwitz--Severi numbers]{On Hurwitz--Severi numbers}

\author[Yu.\,Burman]{Yurii Burman}

\address{National Research University Higher School of Economics
  (NRU-HSE), Russia, and Independent University of Moscow,
  B.~Vlassievskii per., 11, Moscow, 119002, Russia}

\email{burman@mccme.ru}

\author[B.\,Shapiro]{Boris Shapiro}
\address{Department of Mathematics, Stockholm University, S-10691, 
Stockholm, Sweden} 
\email{shapiro@math.su.se}


 \begin{abstract} 
For a point $p\in \CP^2$ and a triple $(g,d,\ell)$ of non-negative integers 
we define a {\em Hurwitz--Severi number} $\h_{g,d,\ell}$ as the number of 
generic irreducible plane curves of genus $g$ and degree $d+\ell$ having an 
$\ell$-fold node at $p$ and at most ordinary nodes as singularities at the 
other points, such that the projection of the curve from $p$ has a 
prescribed set of local and remote tangents and lines passing through 
nodes. In the cases $d+\ell\ge g+2$ and $d+2\ell \ge g+2 > d+\ell$ we 
express the Hurwitz--Severi numbers via appropriate ordinary Hurwitz 
numbers. The remaining case $d+2\ell<g+2$ is still widely open.
 \end{abstract}

\maketitle

\section{Introduction and main results}

In what follows we will always work over the field $\Complex$ of complex 
numbers, and by a genus $g$ of a (singular) curve $C$ we mean its geometric 
genus, i.e.\ the genus of its normalisation. 

Fix a point $p\in \CP^2$ and denote by $\W_{g,d,\ell}$ the set consisting 
of all reduced irreducible plane curves of degree $d+\ell$, genus $g$, 
having an $\ell$-fold node at the point $p$ (i.e.\ $\ell$ smooth local 
branches intersecting transversally at $p$; $\ell = 0$ means that $p$ does 
not belong to the curve), all the singularities outside $p$ (if any) being 
ordinary nodes. The set $\W_{g,d,\ell}$ is nonempty if and only if
 \begin{equation}\label{Eq:Nonempty}
g\le \binom{d+\ell-1}{2}-\binom {\ell}{2},
 \end{equation}
see \cite{Ra1}. 

$\W_{g,d,\ell}$ is usually referred to as the (open, generalized) {\em 
Severi variety}, the classical case corresponding to $\ell=0$. The study of 
this variety was initiated by F.~Severi \cite{Se} back in the 1920s. In a 
number of celebrated papers (see e.g., \cite{Ha1}, \cite{HaMo}, \cite{Ra1}, 
\cite{Ra2}) $\W_{g,d,\ell}$ was proved to be irreducible of dimension 
$3d+2\ell+g-1$. Another well-studied characteristics of the Severi 
varieties is their degree, see \cite{CaHa}.

A Hurwitz--Severi number $\h_{g,d,\ell}$, which we define below, seems to 
be an equally natural characteristics of $\W_{g,d,\ell}$ as its degree, but 
it is apparently more difficult to calculate for general triples 
$(g,d,\ell)$.

The set $\W_{g,d,\ell}$ is acted upon by the group $G \subset
\name{PGL}(3,\Complex)$ of projective transformations of $\CP^2$
preserving $p$ and each line passing through $p$. Obviously, $G$ is a
$3$-dimensional Lie group that acts locally freely on
$\W_{g,d,\ell}$. (In fact, unions of lines passing through $p$ are the
only curves having positive-dimensional stabilizers under this
action.) Denote the orbit space of the action by $\tW_{g,d,\ell}
\bydef \W_{g,d,\ell}/G$; it is smooth almost everywhere and its
dimension equals $3d+2\ell+g-4$.

Let us denote by $\C$ a normalisation of a given plane curve $C$ and by 
$\kappa: \C \to C$, the normalisation map. For a curve $C \in 
\W_{g,d,\ell}$, one defines the associated  meromorphic function of degree 
$d$
 \begin{equation*}
\al_C \bydef \pi_p \circ \kappa: \C \to  p^\perp \simeq \CP^1
 \end{equation*}
obtained by composing the normalisation map with the standard projection 
$\pi_p: \CP^2\setminus p\to p^\perp$ from the point $p$ to the pencil 
$p^\perp \simeq \CP^1$ of lines passing through $p$.

For a generic $C \in \W_{g,d,\ell}$, there are $\ell$ distinct lines 
tangent to $C$ at $p$ ({\em local tangents}), and $2d+2g-2$ distinct  
lines passing through $p$ and tangent to $C$ elsewhere ({\em remote 
tangents}). Additionally, the curve $C$ has 
 \begin{equation*}
\#_{\text{nodes}} = \binom{d+\ell-1}{2}-\binom{\ell}{2}-g = \binom{d-1}{2} 
+ \ell(d-1) - g \ge 0
 \end{equation*}
ordinary nodes (outside $p$), see e.g.\ \cite{OnSh}. A line passing 
through $p$ and a remote node will be called {\em node-detecting}.
 
For any set $X,$ denote by $X^{(m)}$ its $m$-th symmetric power,
i.e.\ the quotient of the Cartesian product $X\times X\times \dots
\times X$ of $m$ copies of $X$ by the natural action of the symmetric
group $S_m$ permuting the copies. In the case when $X$ is a smooth
complex curve, $X^{(m)}$ is naturally interpreted as the set of
effective divisors of degree $m$ on $X$. Below we will denote elements
of $X^{(m)}$ as divisors $z_1 \DT+ z_m$, where $z_1 \DT, z_m \in X$
are not necessarily distinct. The $m$-th symmetric power
$(\CP^1)^{(m)}$ is identified with $\CP^m$ by means of the standard
map sending the divisor $[z_1{:}w_1] \DT+ [z_m{:}w_m]$ to $[a_0
  \DT{{:}} a_m]$, where $\sum_{k=0}^m a_k t^k s^{m-k} \bydef
\prod_{k=1}^m (tz_k-sw_k)$.

We now define three natural maps
 \begin{align*}
&\RemT: \W_{g,d,\ell} \to (p^\perp)^{(2d+2g-2)},\\
&\LocT:\W_{g,d,\ell} \to (p^\perp)^{(\ell)},\\
&\Node: \W_{g,d,\ell} \to (p^\perp)^{(\#\text{nodes})},
 \end{align*}
where $\RemT$ sends $C\in \W_{g,d,\ell}$ to the divisor of its
remote tangents, $\LocT$ sends $C\in \W_{g,d,\ell}$ to the divisor of
its local tangents, and $\Node$ sends $C\in \W_{g,d,\ell}$ to the
divisor of its node-detecting lines. Observe also that, for any $C\in
W_{g,d,\ell}$, $\RemT(C)$ coincides with the divisor of the critical
values of the meromorphic function $\al_\C$.

$\RemT$, $\LocT$ and $\Node$ are obviously preserved by the action of
the group $G$ and, therefore, can be considered as maps defined on
$\tW_{g,d,\ell}$.

The triple of maps 
 \begin{equation*}
\Branch_{g,d,\ell} \bydef (\RemT,\LocT,\Node): \tW_{g,d,\ell} 
\to \cP_{g,d,\ell}
 \end{equation*}
where
 \begin{equation*}
\cP_{g,d,\ell} \bydef (p^\perp)^{(2d+2g-2)} \times (p^\perp)^{(\ell)} 
\times (p^\perp)^{(\#\text{nodes})},
 \end{equation*}
is called the {\em branching morphism}. An easy calculation shows that
for all triples $(g,d,\ell),$ one has
 \begin{align*}
\dim \tW_{g,d,\ell} & =3d+2\ell+g-4 =
2d+2g-2+\ell+\#_{\mathrm {nodes}} - \frac{(d-2)(d+2\ell-3)}{2}\\
& \le 2d+2g-2+\ell+\#_{{\mathrm {nodes}}}=\dim \cP_{g,d,\ell};
 \end{align*}
the equality takes place for the series of triples $(g,2,\ell)$
with $g\le \ell$ (cf.\ \eqref{Eq:Nonempty}) and in two additional
cases: $(0,3,0)$ and $(1,3,0)$.

 \begin{definition}\strut
 \begin{itemize}
\item A triple $(g,d,\ell)$ is called {\em bendable} if $\dim
\tW_{g,d,\ell} \ge 2d+2g-2+\ell$. This is equivalent to
$d+\ell \ge g+2$ and means that $\dim \tW_{g,d,\ell}$ is
larger than or equal to the sum of the number of branch points of
$\al_C$ plus the number of local tangents. A triple $(g,d,\ell)$ is
called {\em strongly bendable} if $\dim \tW_{g,d,\ell} =
2d+2g-2+\ell+\#_{\text{nodes}}$, i.e. $\Branch_{g,d,\ell}$ is the
map between spaces of equal dimension. (All these cases are listed
above.)

\item A triple $(g,d,\ell)$ is called {\em semi-bendable} if $\dim 
\tW_{g,d,\ell} < 2d+2g-2+\ell$, but $\dim \tW_{g,d,\ell} \ge 2d+2g-2$. This 
is equivalent to $d+\ell < g+2 \le d+2\ell$ and means that $\dim 
\tW_{g,d,\ell}$ is larger than or equal to the number of branch points of 
$\al_C$.

\item Otherwise, a triple $(g,d,\ell)$ is called {\em unbendable}. It
means that $\dim \tW_{g,d,\ell} < 2d+2g-2$ or, 
equivalently, that $d + 2\ell < g+2$.
 \end{itemize}
 \end{definition}

 \begin{conjecture}
For any triple $(g,d,\ell)$ which is not strongly bendable, 
$\Branch_{g,d,\ell}$ is a birational map of $\tW_{g,d,\ell}$ on its image.
 \end{conjecture}

In view of this conjecture, to get combinatorially meaningful
quantities, we need to make the image space smaller so that its
dimension would be equal to that of $\tW_{g,d,\ell}$.

 \begin{definition}\strut
 
 \begin{enumerate}
\item Given a bendable triple $(g,d,\ell)$, let $\bar a = a_1 \DT+
  a_{2d+2g-2}$, $\bar b = b_1 \DT+ b_\ell$, and $\bar x = x_1 \DT+
  x_m$ be generic divisors on $p^\perp$ of degrees $2d+2g-2$, $\ell$,
  and $m \bydef \dim \tW_{g,d,\ell} - (2d+2g+\ell-2) =
  d+\ell-g-2$, respectively. Then the {\em Hurwitz--Severi number}
  $\h_{g,d,\ell}$ is defined as the number of $G$-orbits $\OOO \in
  \tW_{g,d,\ell}$ such that $\RemT(\OOO) =\bar a$,
  $\LocT(\OOO) = \bar b$, and $\Node(\OOO) \ge \bar x$ (i.e. all the
  lines $x_1 \DT, x_m$ are node-detecting for any $C \in \OOO$, but
  $C$ may have other node-detecting lines as well).

\item Given a semi-bendable triple $(g,d,\ell)$, let $\bar a = a_1
  \DT+ a_{2d+2g-2}$ and $\bar b = b_1 \DT+ b_m$ be generic divisors on 
  $p^\perp$ of degrees $2d+2g-2$ and $m \bydef \dim 
  \tW_{g,d,\ell} - (2d+2g-2) = d+2\ell-g-2$, respectively. Then 
  the {\em Hurwitz--Severi number} $\h_{g,d,\ell}$ is defined as the number 
  of $G$-orbits $\OOO \in \tW_{g,d,\ell}$ such that $\RemT(\OOO) 
  = \bar a$ and $\LocT(\OOO) \ge \bar b$ (i.e. all lines $b_1 \DT, b_m$ 
  are local tangents for any $C \in \OOO$, but $C$ may have other local 
  tangents as well).

\item Given an unbendable triple $(g,d,\ell)$, let $\bar a = a_1 \DT+
  a_m$ be a generic divisor on $p^\perp$ of degree $m \bydef \dim 
  \tW_{g,d,\ell}$. Then the {\em Hurwitz--Severi number} 
  $\h_{g,d,\ell}$ is defined as the number of orbits $\OOO \in 
  \tW_{g,d,\ell}$ such that $\RemT(\OOO) \ge \bar a$ (i.e. all 
  the lines $a_1 \DT, a_m$ are remote tangents for any $C \in \OOO$).
 \end{enumerate}
 \end{definition}
 
 \begin{Remark}
One can define a branching morphism and a Hurwitz-like number not only for 
Severi varieties, but for many other natural families of plane algebraic 
curves. Given a generic curve $\ga$ in such a family, take the divisor of 
all lines passing through a given point $p\in \CP^2$ which are not in 
general position with respect to $\ga$. Then one can either define a 
branching morphism by just mapping $\ga$ to this divisor or (as we did 
above) one can additionally split this divisor into several subdivisors 
keeping track of different singularities of the intersection of $\ga$ with 
a given line. A Hurwitz-like number will be the number of preimages of a 
generic subspace of appropriate dimension in the image space under the 
branching morphism.
 \end{Remark}

Our main results are formulas for the Hurwitz--Severi numbers in the 
bendable and semi-bendable cases. Consider the set of pairs $(\C,\al)$, 
where $\C$ is a connected smooth curve of genus $g$ and $\al: \C \to \CP^1$ 
is the meromorphic function of degree $d$ with a prescribed set of simple 
critical values. Such pairs are considered up to an isomorphism: $(\C,\al) 
\sim (\C',\al')$ if there exists a holomorphic homeomorphism $\phi: \C \to 
\C'$ such that $\al = \al' \circ \phi$. The number of the pairs is equal to 
$h_{g,1^d}/d!$, where $h_{g,1^d}$ is the ordinary Hurwitz number of genus 
$g$ and partition $1^d = (1 \DT, 1)$ of  $d$; see \cite{KaLa} for the 
precise definition and an algorithm of computation of the Hurwitz numbers.

 \begin{theorem}\label{Th:Bend}
Let $(g,d,\ell)$ be a bendable triple. Then the Hurwitz--Severi number 
$\h_{g,d,\ell}$ is  equal to ${\binom{d}{2}}^{d+\ell-g-2} 
d^{\ell} h_{g,1^d}/d!$.
 \end{theorem}
 
 \begin{theorem}\label{Th:Semi}
Let $(g,d,\ell)$ be a semi-bendable triple. Then the Hurwitz--Severi number 
$\h_{g,d,\ell}$ is  equal to $d^{d+2\ell-g-2} 
\binom{2g-d-\ell-1}{g-3} h_{g,1^d}/d!$.
 \end{theorem}

 \begin{Example}
 \begin{enumerate}

 \item Projection of a smooth cubics from a point not lying on
   it. Here the triple $(g,d,\ell) = (1,3,0)$ is bendable. The
   ordinary Hurwitz number $h_{1,1^3} = 240$ by \cite{KaLa}, so the
   Hurwitz--Severi number $\h_{1,3,0} = 40$; this result was earlier
   obtained in \cite{OnSh}.

 \item Projection of a smooth cubics from a point lying on it. The
   triple $(g,d,\ell) = (1,2,1)$ is bendable. The ordinary Hurwitz
   number $h_{1,1^2} = 1$ by \cite{KaLa}, so the Hurwitz--Severi number
   $\h_{1,3,0} = 1$; this can be checked by a direct computation.

\item Projection of a nodal cubics from an outside point corresponds
  to a bendable triple $(g,d,\ell)=(0,3,0)$. Here $h_{0,1^3}=24$ by
  \cite{KaLa}, implying $\h_{0,3,0}=12$. This answer can be checked
  directly using a computer algebra system.
     
 \item Projection of a nodal cubics from its smooth point corresponds
   to $(g,d, \ell)=(0,2,1)$. The ordinary Hurwitz number is
   $h_{0,1^2}=1$, so the Hurwitz--Severi number is $\h_{0,2,1}=1$, which
   is easily checked by hand.

\item Projection of a smooth quartics from its point corresponds to
  $(g,d,\ell)=(3,3,1)$, a semi-bendable triple. The ordinary
  Hurwitz number computed using the formulas of \cite{KaLa} is
  $h_{3,1^3} = 19680$, so the Hurwitz--Severi number is
  $\h_{g,d,\ell}=3280$.

\item Projection of a smooth quartics from an outside point corresponds to 
  $(g,d,\ell) = (3,4,0)$. This is an unbendable triple not covered by
  Theorems \ref{Th:Bend} and \ref{Th:Semi}. This case was investigated
  by R.~Vakil in \cite{Va} using different technique.

\end{enumerate}
\end{Example}

Theorems~\ref{Th:Bend} and \ref{Th:Semi} give a complete description of 
Hurwitz--Severi numbers in the bendable and semi-bendable cases. Unlike 
them, the unbendable case seems to require completely new ideas. The only 
result in the unbendable case known to the authors at the time of writing 
is \cite{Va}. 

\subsection*{Acknowledgements} The authors are grateful to M.~Feigin,
A.~Gorinov, M.~Kazarian, S.~Lando, S.~Lvovski, M.~Shapiro, and
I.~Tyomkin for discussions. The first-named author wishes to thank the
Stockholm University for its warm hospitality. The research of the
first named author was supported by the Russian Academic Excellence
Project `5-100' and the grant No.~15-01-0031 ``Hurwitz numbers and
graph isomorphism'' of the Scientific Fund of the Higher School of
Economics. The second named author wants to thank the Department of
Mathematics of Higher School of Economics in Moscow for the
hospitality and financial support of his visit in April 2014, when
this project was initiated. The current project is a continuation of
the previous research, see \cite{BuLv} and \cite{OnSh}, which was the
outcome of the previous visit by the second named author to the HSE in
2013.

\section{Proofs} 

The symmetric power $\C^{(k)}$ of the curve $\C$ is the set of effective 
divisors $D$ of degree $k$ on $\C$. Set $D = x_1 c_1 \DT+ x_m c_m \in 
\C^{(k)}$ (with $c_1 \DT, c_m \in \C$, $x_1 \DT, x_m \in \ZZ_{>0}$, and 
$x_1 \DT+ x_m = k$). Introduce a complex coordinate $z_i$ with $z_i(c_i) = 
0$ in an open set $U_i \subset \C$, $c_i \in U_i$, and set $\tilde D = p_1 
\DT+ p_k \in U$, where $U$ is the image of $U_1 \DT\times U_m$ under the 
standard projection $\C^k \to \C^{(k)}$. Without loss of generality, it 
means that $p_1 \DT, p_{k_1} \in U_1$, $p_{k_1+1} \DT, p_{k_2} \in U_2$, 
\dots, $p_{k_{m-1}+1} \DT, p_k \in U_m$ for some $1 \le k_1 \DT\le k_{m-1} 
\le k$. For every $i = 1 \DT, s,$ consider a principal part $F_i$ of a 
meromorphic function $f_i: U_i \to \CP^1$ having at $p_i$ a pole of the 
degree not exceeding the multiplicity of $p_i$ in the divisor $\tilde D$ 
and having no other poles in $U_i$; let $F = (F_1 \DT, F_s)$.  
 For $c \in \C,$ using the above local coordinates, one has 
 \begin{multline*}
F(c) = \frac{a_1 + a_2 z_1(c) \DT+ a_{k_1} 
z_1(c)^{k_1-1}}{(z_1(c)-z_1(p_1)) \dots (z_1(c)-z_1(p_{k_1}))} + \dots\\
+\frac{a_{k_{m-1}+1} + a_{k_{m-1}+2}z_m(c) \DT+ a_k 
z_m(c)^{k-k_{m-1}-1}}{(z_m(c)-z_m(p_{k_{m-1}+1})) \dots (z_m(c)-z_m(p_k))}.
 \end{multline*}
So, the vectors $F$ of principal parts form a rank $k$ vector bundle on 
$\C^{(k)}$; the coefficients $a_1 \DT, a_k$ form its trivialisation over 
the set $U$. An immediate comparison of the transition maps shows that this 
bundle is isomorphic to the tangent bundle $T\C^{(k)}$.

Given a $1$-form $\nu$ holomorphic in $U_i$, we define a linear functional 
$\nu_z$ on the space of principal parts by the formula
 \begin{equation*}
\nu_z(F_i) = \Res_z F_i\nu.
 \end{equation*}
For a divisor $D = x_1 c_1 \DT+ x_m c_m \in \C^{(k)},$ define $\nu_D \bydef 
\sum_{i=1}^m \nu_{c_i}$. It follows from the above reasoning that $\nu_D$ 
is a section of the complex cotangent bundle $T^*\C^{(k)}$; cf.\ the fiber 
bundle of principal parts introduced in \cite{ELSV}.

There exists a natural map $\Phi: \OO(D) \to T_D\C^{(k)}$ sending a 
memomorphic function $f \in \OO(D)$ to the $m$-tuple $F = (F_1 \DT, F_m)$ of 
its principal parts at the points $c_1 \DT, c_m$. By the Riemann--Roch 
theorem, $F = \Phi(f)$ for some $f$ if and only if $\nu_D(F) = 0$ for every  
holomorphic $1$-form $\nu$ on $\C$. Fixing a basis $\nu_1 \DT, \nu_g$ of 
holomorphic $1$-forms on $\C$, we can calculate the dimension $h^0(D) = 
\dim \OO(D)$ as $k - \dim \langle (\nu_1)_D \DT, (\nu_g)_D\rangle$.

To prove our main results, we need the following technical statement which 
is apparently well-known to the specialists, but we could not find it 
explicitly in the literature:

  \begin{proposition}\label{Pr:NonSp}
Take a pair $(\C,\al)$, where $\C$ is a smooth curve of genus $g$
and $\al: \C \to \CP^1$ is a meromorphic function of degree $d$,
and suppose that $(\C,\al)$ is generic among such pairs. Set $D_\al
:= z_1 \DT+ z_d$, where all  $z_i$ are pairwise distinct. If
$m \ge g+2 \ge d$ and $z_{d+1} \DT, z_m$ are generic pairwise distinct
points, then the divisor $z_1 \DT+ z_m$ is non-special.
  \end{proposition}

Generic divisors are never special, but $z_1 \DT+ z_m$ may be non-generic 
because $z_1 \DT+ z_m \ge D_\al$ for some $\al$ of degree $d$.

 \begin{proof}
Let $\nu_1 \DT, \nu_g$ be a basis of holomorphic $1$-forms on $\C$.
Since $z_1 \DT+ z_d = D_\al$, one has $h^0(z_1 \DT+ z_d) \ge 2$. If
there exists $\psi \in \OO(D_\al)$ not proportional to $\al$,
then there exists their non-constant linear combination with no pole
at $z_d$. Since a generic $d$-gonal curve $\C$ is not $(d-1)$-gonal
\cite{Farkas}, this is impossible, and therefore $h^0(z_1 \DT+ z_d) =
2$.

We now prove by induction that if $d+s \le g+1$ and the points $z_{d+1}
\DT, z_{d+s}$ are in general position, then
 \begin{equation*}
\dim\langle (\nu_1)_{z_1 \DT+ z_{d+s}} \DT, (\nu_g)_{z_1 \DT+ 
z_{d+s}}\rangle = d+s-1.
 \end{equation*}  
Assume that starting with some $s$ the statement fails. It means that
for $z_{d+1} \DT, z_{d+s-1}$ in general position,
 \begin{equation*}
\dim\langle (\nu_1)_{z_1 \DT+ z_{d+s-1}} \DT, (\nu_g)_{z_1 \DT+ 
z_{d+s-1}}\rangle = d+s-2,
 \end{equation*}  
but there exists a non-empty open set $\Omega \subset \C$ such that if
$z_{d+s} \in \Omega$, then
 \begin{equation*}
\dim\langle (\nu_1)_{z_1 \DT+ z_{d+s}} \DT, (\nu_g)_{z_1 \DT+ 
z_{d+s}}\rangle = d+s-2
 \end{equation*}  
as well. In other words,  vector $\vec\nu_{z_{d+s}} \bydef
((\nu_1)_{z_{d+s}} \DT, (\nu_g)_{z_{d+s}})$ is a linear combination of
$\vec\nu_{z_i}$, $i = 1 \DT, d+s-1$.

For an arbitrary positive integer $q$, consider  $q$ points $z_{d+s} 
\DT, z_{d+s+q-1} \in \Omega$. Then for $j = d+s \DT, d+s+q-1$, every 
$\vec\nu_{z_j}$ is a linear combination of $\vec\nu_{z_i}$, $i = 1 \DT, 
d+s-1$, and therefore
 \begin{equation*}
\dim\langle (\nu_1)_{z_1 \DT+ z_{d+s+q-1}} \DT, (\nu_g)_{z_1 \DT+ 
z_{d+s+q-1}}\rangle = d+s-2
 \end{equation*}  
for any $q$. Hence the divisor $z_{d+s} \DT+ z_{d+s+q-1}$ is special for 
any collection $z_{d+s} \DT, z_{d+s+q-1} \in \Omega$, and therefore the set 
of special divisors of any degree $q > g$ on $\C$ contains an open subset 
$\Omega^{(q)} \subset \C^{(q)}$. The latter claim is false since the set of 
special divisors of any sufficiently large degree is nowhere dense.
 \end{proof}

 \begin{proof}[Proof of Theorem~\ref{Th:Bend}]
Take $p = [0{:}1{:}0]$ and suppose without loss of generality that a
curve $C \in \W_{g,d,\ell}$ does not contain the point
$[1{:}0{:}0]$. Then the normalisation map $\kappa: \C \to C$ is given
by
 \begin{equation*}
\kappa(z) = [\beta(z){:}1{:}\al(z)],
 \end{equation*}
where $\al,\beta: \C \to \CP^1$ are meromorphic functions of
degrees $d$ and $d+\ell$, respectively, such that $D_\beta \ge
D_\al$, i.e. $D_\beta - D_\al$ is an effective divisor on
$\C$.

Take a generic divisor $\bar a = a_1 \DT+ a_{2d+2g-2}$ on $\CP^1$. As it 
was noted above, there exist $h_{g,1^d}/d!$ pairs $(\C,\al)$ such that 
$\bar a$ is the divisor of the critical values of $\al$. For any $\beta$, 
one can take $C \bydef \kappa(\C)$, with the map $\kappa: \C \to \CP^2$ as 
above. Then $\RemT(C) = \bar a$ regardless of the choice of $\beta$. Set 
$D_\al := z_1 \DT+ z_d$ and notice that in general position the points $z_1 
\DT, z_d \in \C$ are pairwise distinct (i.e.\ $\al$ has only simple poles).

Now take a generic divisor $\bar b = b_1 \DT+ b_\ell$ on $\CP^1$ and 
choose points $z_{d+1} \DT, z_{d+\ell} \in \C$ such that $\bar b = 
\al(z_{d+1}) \DT+ \al(z_{d+\ell})$. Since the degree of $\al$ is 
$d$, there are $d^\ell$ ways to do this; the points $z_{d+1} \DT, 
z_{d+\ell}$ are pairwise distinct if $\bar b$ is in general position. This 
guarantees the equality $\LocT(C) = \bar b$ for the curve $C = \kappa(\C)$, 
provided that $D_\beta = z_1 \DT+ z_{d+\ell}$.

Assume now that $\bar x = x_1 \DT+ x_{d+\ell-g-2}$ is a generic
divisor on $\CP^1$. For each $i$, take a pair of points $u_i
\ne v_i \in \C$ such that $\al(u_i) = \al(v_i) = x_i$; there are
$\binom{d}{2}^{d+\ell-g-2}$ ways to do this. For $i = 1, \dots,
d+\ell-g-2$, define the functionals $\rho_i: \OO(z_1 \DT+ z_{d+\ell})
\to \Complex$ by
 \begin{equation*}
\rho_i(\beta) \bydef \beta(u_i)-\beta(v_i).
 \end{equation*}
Apparently, $\rho_i(\beta) = 0$ if and only if the line $x_i$ is
node-detecting for the curve $C = \kappa(\C)$.

 \begin{lemma}
For a generic choice of $\al$ and $z_{d+1}, \dots, z_{d+\ell},$ the
functionals $\rho_i$, $i = 1 \DT, d+\ell-g-2$ are linearly
independent.
 \end{lemma}

 \begin{proof}
By the Riemann--Roch theorem, $h^0(z_1 \DT+ z_{d+\ell}) \ge
d+\ell-g+1$. Thus, for any $k \le d+\ell-g-2,$ there exists a function
$\beta \in \OO(z_1 \DT+ z_{d+\ell})$ such that $\rho_1(\beta) \DT=
\rho_{k-1}(\beta) = 0$. If for generic $\al$ and $z_{d+1}, \dots,
z_{d+\ell},$ the functional $\rho_k$ is a linear combination of
$\rho_i$, $1 \le i \le k-1$, then there exists an open subset $\Omega
\subset \C$ with the following property. If $\al(u) = \al(v) = x
\in \Omega,$ then $\beta(u) = \beta(v)$, implying that $\beta(z)$ is a
function of $\al(z)$ for $z \in \al^{-1}(\Omega)$. Therefore,
for any $z \in \al^{-1}(\Omega)$, the line $\al(z) \in p^\perp$
intersects the curve $C$ at exactly one point. Since
$\al^{-1}(\Omega)$ is open, there exists $z_* \in
\al^{-1}(\Omega)$ such that the intersection is transversal. But if
$d > 1$, this is impossible.
 \end{proof}

Now by Proposition \ref{Pr:NonSp}, one has generically
 \begin{equation*}
h^0(z_1 \DT+ z_{d+\ell}) = d+\ell-g+1
 \end{equation*}
implying that the solutions of the equations $\rho_1(\beta) = \dots =
\rho_{d+\ell-g-2}(\beta) = 0$ form a $3$-dimensional space. The group
$G$ acts transitively upon this space, and Theorem \ref{Th:Bend}
follows.
 \end{proof}

The proof of Theorem~\ref{Th:Semi} is based on the following statement. 

 \begin{proposition} \label{Pp:Intersect}
Let $(g,d,\ell)$ be a semi-bendable triple. Then for a generic pair
$(\C,\al)$, where $\C$ is a smooth curve and $\al: \C \to
\CP^1$ is a degree $d$ meromorphic function, and for a generic
divisor $D_0$ of degree $d+2\ell-g-2$ on $\C$, the set $\cD
\bydef \{D \in \C^{(g+2-d-\ell)} \mid h^0(D_\al+D_0+D) \ge 3\}$ is
finite and contains $\binom{2g-d-\ell-1}{g-3}$ elements. Additionally,
for any $D \in \cD,$ one has $h^0(D_\al+D_0+D) = 3$.
 \end{proposition}

 \begin{proof}
Choose $D \in \cD$ and $\beta \in \OO(D_\al+D_0+D)$ and define a 
plane curve $C_{D,\beta} \bydef \kappa(\C)$, where $\kappa(z) \bydef 
[\beta(z){:}1{:}\al(z)]$. The group $G$ acts on $\OO(D_\al+D_0+D)$ by  
$\beta \mapsto a\beta + b\al + c$, where $a,b,c \in \Complex, a \ne 0$. 
Thus $G$ acts on the set of all $C_{D,\beta}$ with a fixed $D$.

Prove first that $\cD$ is finite and that $h^0(D_\al+D_0+D) = 3$ for any $D 
\in \cD$. Consider the orbit space $\widetilde{\cD} \bydef \{(D,\beta) \mid 
D \in \cD, \beta \in \OO(D_\al+D_0+D)\}/G$ and let $\delta \bydef \dim 
\widetilde{\cD}$ be its dimension. The pair $(\C,\al)$ is determined, up to 
a finite choice, by the divisor of the critical values of $\al$, which has 
degree $2d+2g-2$; so the set of all such pairs has dimension $2d+2g-2$. The 
dimension of the set of all divisors $D_0$ is equal to $\deg D_0 = 
d+2\ell-g-2$. The choice of $(D,\beta) \in \widetilde{\cD}$ determines a 
curve $C_{D,\beta}$ up to the action of $G$, that is, it determines a point 
in $\tW_{g,d,\ell}$. On the other hand, for a given curve $C_{D,\beta} \in 
\W_{g,d,\ell}$, one can uniquely restore the divisor $D$ on the 
normalisation $\C$ of $C_{D,\beta}$ noticing that its points are the poles 
of $\beta$ or, equivalently, the points of $\C$ sent by the normalisation 
map to the base point $p \in C_{D,\beta}$. So, different choices of $D \in 
\cD$ and different orbits of the $G$-action on $\OO(D_\al+D_0+D)$ for a 
fixed $D$ correspond to different points of $\tW_{g,d,\ell}$. This implies 
the inequality
 \begin{equation*}
\dim \tW_{g,d,\ell} \ge (2d+2g-2) + (d+2\ell-g-2) + \delta.
 \end{equation*}
Since $\dim \tW_{g,d,\ell} = 3d+2\ell+g-4$ (see e.g., \cite{Ha1}), one
gets that $\delta = 0$. Thus, $\cD$ consists of a finite number
of points, and for any such point $D,$ the number of $G$-orbits in
$\OO(D_\al+D_0+D)$ is finite, which means that $h^0(D_\al+D_0+D)
= 3$.

Count now the points $D \in \cD$. Set $D_\al \bydef z_1 \DT+ z_d$, 
$D_0 \bydef z_{d+1} \DT+ z_{2d+2\ell-g-2}$, $D \bydef z_{2d+2\ell-g-1} \DT+ 
z_{d+\ell}$ and denote
 \begin{equation*}
D' \bydef D_\al+D_0+D-z_d = z_1 \DT+ z_{d-1} + z_{d+1} \DT+ z_{d+\ell}.
 \end{equation*}
As was shown above, $h^0(D_\al+D_0+D) \ge 3$ if and only if
 \begin{equation*}
\dim \langle (\nu_1)_{D_\al+D_0+D} \DT,
(\nu_g)_{D_\al+D_0+D}\rangle \le d+\ell-2
 \end{equation*}
or, equivalently, $\dim \langle \vec\nu_{z_1} \DT, 
\vec\nu_{z_{d+\ell}}\rangle \le d+\ell-2$. $\C$ is a generic $d$-gonal 
curve, therefore it is not $(d-1)$-gonal implying that  $h^0(D_\al) = 2$. 
Thus the vector $\vec\nu_{z_d}$ is a linear combination of $\vec\nu_{z_1} 
\DT, \vec\nu_{z_{d-1}}$ (see the proof of Proposition \ref{Pr:NonSp} for 
notation), which means that the last condition is equivalent to
 \begin{equation*}
\dim \langle \vec\nu_{z_1} \DT,
\vec\nu_{z_{d-1}}, \vec\nu_{z_{d+1}} \DT, \vec\nu_{z_{d+\ell}}\rangle
\le d+\ell-2,
 \end{equation*}
i.e. $\dim \langle (\nu_1)_{D'} \DT, (\nu_g)_{D'}\rangle \le
d+\ell-2$.

For any $k$ denote by $S_k: \C^k \to \C^{(k)}$ the natural projection; for 
any vector $X \in \C^k$ denote by $\iota_X: \C^k \to \C^{m+k}$ the natural 
embedding (coordinates of $X$ are written before the coordinates of the 
argument). Take any point $Z = (z_1 \DT, z_{d-1}, z_{d+1} \DT, 
z_{2d+2\ell-g-2})$ such that $S_{2d+2\ell-g-3}(Z) = D_0 + D_1 - z_d$  and 
consider the vector bundle $E = \iota_Z^* S_{d+\ell-1}^* T^* 
\C^{(d+\ell-1)}$ of rank $d+\ell-1$ on $\C^{g+2-d-\ell}$. (In other words, 
$Z$ is an arbitrary ordering of $z_1 \DT, z_{d-1}, z_{d+1}, 
\allowbreak\dots, z_{2d+2\ell-g-2}$.) The Riemann--Roch theorem implies 
that $D \in \cD$ if and only if for any $W \in S_{g+2-d-\ell}^{-1}(D_1)$, 
one has
 \begin{equation*}
\dim \langle \iota_Z^* S_{d+\ell-1}^*(\nu_1)_{D'}(W) \DT, \iota_Z^*
S_{d+\ell}^*(\nu_g)_{D'}\rangle \le d+\ell-2.
 \end{equation*}
Since we have shown that  variety $S_{g+2-d-\ell}^{-1}(\cD)$ is 
$0$-dimensional, its number of points is given by the Porteous formula 
\cite{Porteous}:
 \begin{equation}\label{Eq:Porteous}
\# S_{g+2-d-\ell}^{-1}(\cD) Y = \det \left(\begin{array}{cccll}
c_1(E) & c_2(E) & \dots & c_{g+1-d-\ell}(E) & c_{g+3-d-\ell}(E) \\
1 & c_1(E) & \dots & c_{g-d-\ell}(E) & c_{g+1-d-\ell}(E)\\
0 & 1 & \dots & c_{g-1-d-\ell}(E) & c_{g-d-\ell}(E)\\
\hdotsfor{5}\\
0 & 0 & \dots & \hphantom{c_1(E)}1 & c_1(E)
\end{array}\right),
 \end{equation}
where $Y \in H^{2(g+2-d-\ell)}(\C^{g+2-d-\ell})$ is the generator of the 
top-dimensional cohomology, i.e. it is the Poincar\'e dual of a point. 

Set $m \bydef g+2-d-\ell$ for brevity, and denote by $\cN_{k,m}$ the
lower right $(k \times k)$-minor of \eqref{Eq:Porteous}. In particular,
$\# S_{g+2-d-\ell}^{-1}(\cD) Y = \cN_{m,m}$. Developing the
determinant by its first column, one obtains $\cN_{m,m} =
\sum_{k=1}^m(-1)^{k+1} c_k(E) \cN_{m-k,m}$, implying
 \begin{equation}\label{Eq:NViaC}
\sum_{k=0}^m \cN_{k,m} = \left(\sum_{k=0}^m (-1)^k c_k(E)\right)^{-1}.
 \end{equation}

Denote by $x \in H^2(\C^{(k)})$ the class dual to the fundamental
homological class of the diagonal $\{kz\mid z \in \C\} \subset
\C^{(k)}$.  It follows from the general formula of \cite{Ohmoto} that
$c_k(T^* \C^{(m)}) = \binom{2g-2+k-m}{k} x^k$. One has $S_m^* x = x_1
\DT+ x_m \bydef X$, where $x_i \in H^2(\C^m)$ is the class dual to the
fundamental class of the $i$-th copy of $\C$ in the product $\C
\DT\times \C = \C^k$. Additionally,  one has $\iota_Z^* x_i = x_{i-m} \in
H^2(C^k)$ if $i > m$ and $\iota_Z^* x_i = 0$ if $i \le m$. Therefore
$c_k(E) = \binom{g-3+m+k}{k} X^k$, implying that

 \begin{equation*}
\sum_{k=0}^m (-1)^k \binom{g-3+m+k}{k} X^k = (1+X)^{-(g-3+m)}.
 \end{equation*}
Thus,  \eqref{Eq:NViaC} implies that $\sum_{k=0}^m \cN_{k,m} =
(1+X)^{(g-3+m)}$, giving $\cN_{m,m} = \binom{g-3+m}{g-3} X^m$. Since
$Y = X^m/m!$, then $\# S_m^{-1}(\cD) = m!
\binom{g-3+m}{g-3}$. By dimensional reasons, in generic situation all
the elements $\cD$ are sums of exactly $m$ distinct points, meaning 
that
 \begin{equation*}
\# \cD = \frac{1}{m!} \# S_m^{-1}(\cD) =
\binom{g-3+m}{g-3} = \binom{2g-1-d-\ell}{g-3}.
 \end{equation*}
 \end{proof}

 \begin{proof}[Proof of Theorem~\ref{Th:Semi}]
Similarly to the proof of Theorem \ref{Th:Bend}, for a generic
divisor $\bar a = a_1 \DT+ a_{2d+2g-2}$, there are $h_{g,1^d}/d!$ ways to
choose a  curve $\C$ of genus $g$ and a degree $d$ meromorphic function
$\al: \C \to \CP^1$ such that $\bar a$ is its divisor of critical
values.

Let $D_\al$ be the pole divisor of $\al$; choose
$d+2\ell-g-2$ points $z_{d+1} \DT, z_{2d+2\ell-g-2} \in \C$ such that
$\bar b = \al(z_{d+1}) \DT+ \al(z_{2d+2\ell-g-2})$. For generic
$\bar b$, there are $d^{d+2\ell-g-2}$ ways to do that.

Similar to the proof of Proposition \ref{Pp:Intersect}, denote $D_0
\bydef z_{d+1} \DT+ z_{2d+2\ell-g-2}$ for short, and  denote by
$\cD \subset \C^{(g+2-d-\ell)}$ the set of effective divisors
$D \bydef z_{2d+2\ell-g-1} \DT+ z_{d+\ell}$ of degree $g+2-d-\ell$
such that $h^0(D_\al+D_0+D) \ge 3$. By Proposition
\ref{Pp:Intersect}, the set $\cD$ is finite, and for any $D \in
\cD$, the space $\OO(D_\al+D_0+D)$ has dimension $3$ and
contains exactly one orbit of the group $G$. Therefore the number of
points $C \in \tW_{g,d,\ell}$ with $\RemT(C) = a$ and $\LocT(C) \ge b$
is equal to
 \begin{equation*}
d^{d+2\ell-g-2} h_{g,1^d}/d! \# \cD =
d^{d+2\ell-g-2} \binom{2g-d-\ell-1}{g-3} h_{g,1^d}/d!.
 \end{equation*}
 \end{proof}

\section{Final Remarks}

\def \thesubsubsection {\arabic{subsubsection}}

\subsubsection{} The definition of Hurwitz--Severi numbers given above can 
be easily extended from the class of Severi varieties $W_{g,d,\ell}$ to a 
somewhat broader class $W_{g,d,\ell,\mu}$ which appeared earlier in several 
papers of J.~Harris and Z.~Ran. Namely, one can additionally require that 
curves under consideration have a given set $\mu$ of tangency 
multiplicities to a given line passing through the point $p$. One might 
expect Theorems~\ref{Th:Bend} and \ref{Th:Semi} to have straightforward 
analogs in this more general setup.

\subsubsection{} The problem of calculation of Hurwitz--Severi numbers for 
the simplest unbendable case $g = (d-1)(d-2)/2, \ell = 0$, i.e.\ when a 
smooth plane curve of degree $d$ is projected from a point not lying on it, 
bears a strong resemblance with the problem of calculation of Zeuten 
numbers, see \cite{Ze}. Namely, in a special case, Zeuten's problem asks 
how many smooth plane curves of degree $d$ are tangent to a given set of 
$\frac{d(d+3)}{2}$ lines in general position. The Hurwitz--Severi number 
for the case $g = (d-1)(d-2)/2$ and $\ell = 0$ counts the number of 
$G$-orbits of smooth curves of degree $d$ which are tangent to a given set 
of $\frac{d(d+3)}{2}-3$ generic lines passing through a given point $p$. To 
the best of our knowledge, both problems are unsolved at present and 
apparently are quite difficult.

\subsubsection{} One possible approach to the calculation of Hurwitz 
numbers in the unstable case (such as $\left((d-1)(d-2)/2, d, 0\right)$) 
might be the use of tropical algebraic geometry. For example, in 
\cite{BBM} the authors studied tropical analogs of Zeuten numbers and were 
able to recover some of the classical Zeuten numbers through their tropical 
analogs.

\subsubsection{} It would be interesting to study possible relation of the 
above Hurwitz--Severi numbers to appropriate Gromov-Witten invariants of 
plane curves.

\end{document}